\documentclass[12pt, reqno]{amsart}
\usepackage{color}
\usepackage{hyperref}
\usepackage{graphicx}
\usepackage{amsmath}
\usepackage{amsthm}
\usepackage{amssymb,bbm}%
\usepackage{mathrsfs}
\usepackage{enumerate}
\usepackage{bbm}
\usepackage{color}
\usepackage{hyperref}
\usepackage{graphicx}
\usepackage{mathrsfs}
\usepackage{enumerate}
\usepackage{bbm}

\allowdisplaybreaks

  \evensidemargin
\oddsidemargin
\numberwithin{equation}{section}

\newcommand{\E}{\mathbb E}

\newcommand{\vol}{\operatorname{Vol}}

\newcommand{\R}{\mathbb{R}}

\theoremstyle{plain}
\newtheorem{theorem}{Theorem}[section]
\newtheorem{proposition}[theorem]{Proposition}
\newtheorem{lemma}[theorem]{Lemma}
\newtheorem{corollary}[theorem]{Corollary}
\theoremstyle{definition}

\newtheorem{remark}[theorem]{Remark}

\newcommand{\Cc}{\mathcal{C}}

\begin{document}

\title[No repulsion between Critical Points]{No repulsion between critical points for planar Gaussian random fields}


\author{Dmitry Beliaev}
\email{dmitry.belyaev@maths.ox.ac.uk}
\address{Mathematical Institute, University of Oxford}

\author{Valentina Cammarota}
\email{valentina.cammarota@uniroma1.it}
\address{Department of Statistics, Sapienza University of Rome}

\author{Igor Wigman}
\email{igor.wigman@kcl.ac.uk}
\address{Department of Mathematics, King's College London}




\date{\today}

\begin{abstract}
We study the behaviour of the point process of critical points of isotropic stationary Gaussian fields. We compute the main term in the asymptotic expansion of the two-point correlation function near the diagonal. Our main result implies that for a `generic' field the critical points neither repel nor attract each other. Our analysis also allows to study how the short-range behaviour of critical points depends on their index.
\end{abstract}

\maketitle

\section{Introduction}

\subsection{Two-point correlation function for critical points of planar random fields}

The number of critical points of a function
and their positions are its important qualitative descriptor,
and their study is an actively pursued field of research within a wide range of disciplines, such as
classical analysis (see e.g. ~\cite{Logunov-Sodin etc}), probability (e.g. ~\cite{CamMarWig,CamWig}),
mathematical and theoretical physics (\cite{Hitchin}), spectral geometry (e.g. ~\cite{NS2009,JN}), and cosmology and the study of Cosmic Microwave Background (CMB) radiation (e.g ~\cite{CCFMS}).
In case $F:\R^{2}\rightarrow\R$ (or, more generally, $F:\R^{d}\rightarrow\R$, $d\ge 2$)
is a smooth Gaussian random field, then its set of critical points $\Cc_{F}$ is
a {\em point process} on $\R^{2}$ (resp. $\R^{d}$). If we assume in addition that $F$ is stationary, then it
is possible to employ the Kac-Rice method in order to obtain that, under some mild non-degeneracy assumptions
on $F$ and its mixed derivatives up to $2$nd order, the expected number of critical points lying in a ball or radius $R$, $B(R)\subseteq \R^{2}$
is given precisely by
\begin{equation*}
\E[\# ( \Cc_{F}\cap B(R))] = c_{F}\cdot \vol(B(R)),
\end{equation*}
where $c_{F}>0$ is a constant that could be expressed in terms of some derivatives of the covariance function of $F$
evaluated at the diagonal.

It is then compelling to study the law of $\Cc_{F}$ in more depth, e.g., the variance of $$\# \Cc_{F}(B(R)):=\#(\Cc_{F}\cap B(R)),$$
and the relative positions of critical points, e.g. their attraction and repulsion. Let $$K_{2}(x-y)=K_{2}(x,y)$$
be the $2$-point correlation function of the point process $\Cc_{F}$ (or other point processes) defined as
\begin{equation*}
K_{2}(x-y) = \lim\limits_{\epsilon_{1},\epsilon_{2}\rightarrow 0}
\frac{1}{\vol(B(\epsilon_{1}))\cdot\vol(B(\epsilon_{2}))}\E[\#\Cc_{F}(B_{x}(\epsilon_{1}))\cdot \#\Cc_{F}(B_{y}(\epsilon_{2}))],
\end{equation*}
where $B_x(\epsilon)$ is the radius-$\epsilon$ ball centred at $x$. Note that $K_2$ clearly depends on $F$ but we omit this to simplify the notation. The corresponding field will be always clear from the context. 

Given $K_2$ we immediately get the formula for the second factorial moment of the number of critical points via
\begin{equation}
\label{eq:fact mom Kac Rice}
\E[\#\Cc_{F}(B(R))\cdot (\# \Cc_{F}(B(R))-1)] = \int\limits_{B(R)\times B(R)} K_{2}(x,y)dxdy.
\end{equation}
For $\Cc_{F}$ (and other point processes that are zeros of random Gaussian fields, with $\Cc_{F}$ being the zero set of $\nabla F$),
one can usually derive the $2$-point correlation function via the Kac-Rice formula
\begin{equation}
\label{eq:K2 Kac Rice}
K_{2}(x,y) = \phi_{(\nabla F (x),\nabla F(y))}(0,0)\cdot
\E\left[\left|\det{H_{F}(x)}\cdot \det{H_{F}(y)}\right|\big| \nabla F(x)=\nabla F(y)=0\right],
\end{equation}
where $\phi_{(\nabla F(x),\nabla F(y))}(\cdot)$ is the density of the Gaussian vector $(\nabla F(x),\nabla F(y)) \in \R^{2}\times\R^{2}$,
and $H_{F}(\cdot)$ is the Hessian of $F$ (we exclude the diagonal $\{x=y\}$ which does not contribute to the integral, but the near-diagonal behaviour when $x\approx y$ will play a crucial role). The function \eqref{eq:K2 Kac Rice} is, in turn,
a semi-explicit function of the covariance function of $F$ and its couple of mixed derivatives.

If, in addition, $F$ is assumed to be {\em isotropic}, then $K_{2}(x,y)$ is a function of the Euclidean distance $r=\|x-y\|$.
In many cases, when the covariance function of $F$ is decaying sufficiently rapidly,
the {\em long range} asymptotics of $K_{2}(r)$, $r\rightarrow\infty$ yields the asymptotic variance of the number of critical points
in large balls $B(R)$, $R\rightarrow\infty$, see e.g. ~\cite{CamMarWig,CamWig}, and other quantities, such as the nodal length
of $F$ ~\cite{Berry,Wig}. On the contrary, the {\em short range} asymptotics of $K_{2}(r)$, $r\rightarrow 0$ yields the asymptotic
law of the second factorial moment of the number of critical points of $F$ belonging to {\em small} balls $B(r)$, $r\rightarrow 0$,
again via \eqref{eq:fact mom Kac Rice}. Informally, the probability that there is one critical point in a ball $B(r)$ of small radius $r>0$ is approximately $cr^2$, whereas the probability that there are two critical points in $B(r)$ is
approximately $\iint\limits_{B(r)\times B(r)} K_2(x,y)dxdy$.
If $K_2(r)\to \infty$ as $r\rightarrow 0$, then the probability to have two critical points in $B(r)$ is much higher than the square of the probability to have one critical point in $B(r)$. In this case we say that the critical points \emph{attract} each other. Otherwise, if
$K_2(r)\to 0$ as $r\rightarrow 0$, then the probability to have two critical points in $B(r)$ is much lower than the square of the probability to have one such point. In this case we say that the critical points \emph{repel} each other.

\vspace{2mm}

The first relevant result \cite{BCW2017} was obtained in 2017 when we analysed the asymptotic behaviour of $K_2$ for a particular Gaussian field:  the random monochromatic isotropic plane waves, also referred to as ``Berry's Random Wave Model" (RWM). This field is of a particular interest since it is believed to represent the (deterministic) Laplace eigenfunctions on ``generic" chaotic surfaces, in the high energy limit ~\cite{Berry77}. The RWM is the stationary isotropic random field $F:\R^{2}\rightarrow\R$, uniquely defined by the covariance
function
\[
C_{F}(x):=\E[F(y)\cdot F(x+y)]=J_{0}(\|x\|),
\]
with $J_{0}(\cdot)$ the Bessel $J$ function of the first order.

\begin{figure}
\includegraphics[width=0.27\textwidth]{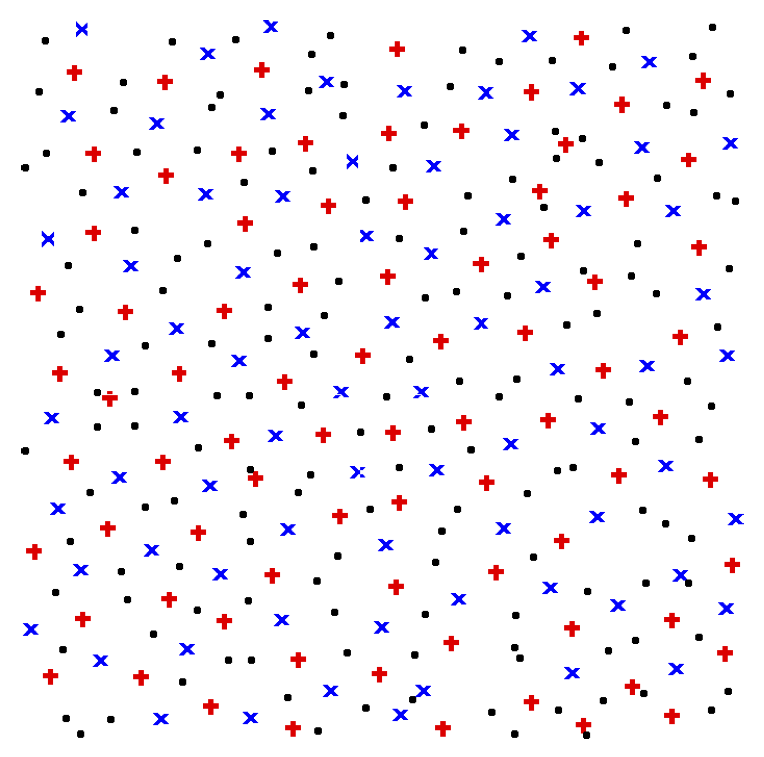}\hspace{0.05\textwidth}
\includegraphics[width=0.27\textwidth]{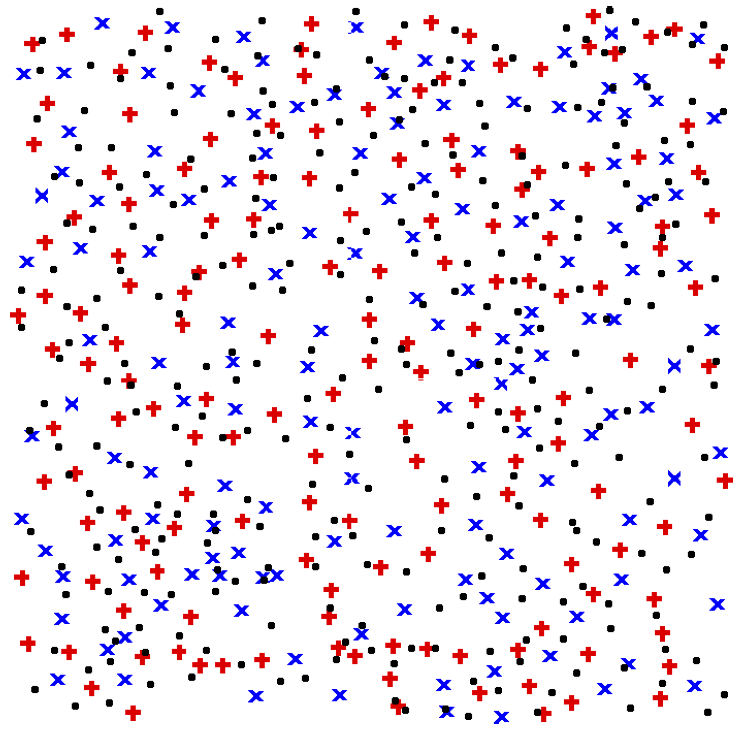}\hspace{0.05\textwidth}
\includegraphics[width=0.27\textwidth]{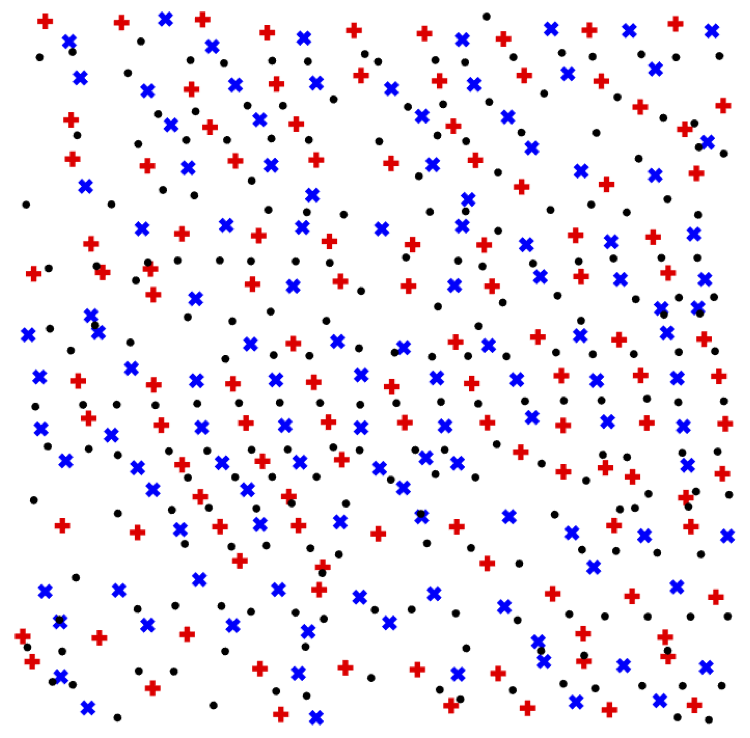}
\caption{Critical points for the random plane wave (left), Bargmann-Fock field (center) and an anisotropic field. Minima and maxima are red and blue plusses and crosses, critical points are black dots. }
\label{fig: critical points}
\end{figure}

The work \cite{BCW2017} was motivated by Figure \ref{fig: critical points} (left) which apparently shows that the critical point repel each other.  It was found  \cite[p. 10]{BCW2017} that for RWM the two-point correlation function $K_{2}^{\textrm{RWM}}(r)$ has the following asymptotic around the diagonal to
\vspace{-0.1cm}
\begin{equation}
\label{eq:K2 asymp BCW}
K_{2}^{\textrm{RWM}}(r) = \frac{1}{2^{5}3\sqrt{3}\pi^{2}}+O_{r\rightarrow 0}(r^{2}),
\end{equation}
\vspace{-0.1cm}
so that, in particular, the critical points of $F$ exhibit no repulsion nor attraction. It was then inferred that the seemingly visible repulsion on some numerically generated pictures could be attributed to {\em rigidity} of critical points, a notion cardinally different from repulsion. The work ~\cite{BCW2017} also allowed for the separation of the critical points into maxima, minima and saddles, and studied the effect of such a separation on the corresponding $2$-point correlation function, resulting in some cases in qualitatively different
behaviour to \eqref{eq:K2 asymp BCW}.

It is then natural to inquire about the analogous question for {\em other} Gaussian random fields, i.e. for the asymptotic law of the $2$-point-correlation function
around the diagonal for {\em other} Gaussian random fields. In particular, whether it is true, that for a generic stationary field, the critical points nether attract nor repel each other.
That critical points do not attract was resolved by Muirhead \cite{Muirhead2019}, who, among other things, proved that there is no attraction for `generic' stationary planar Gaussian random fields, without assuming that the underlying random field is isotropic.

Our first principal Theorem \ref{thm:exp 2 corr main} below yields that
for a Gaussian isotropic random field $F$ satisfying some generic assumptions, one has
\vspace{-0.1cm}
\begin{equation}
\label{eq: K_2 asymptotic}
K_{2}(r) = a_{F}+O_{r\rightarrow 0}(r)
\end{equation}
\vspace{-0.1cm}
with $a_{F}>0$ explicitly evaluated in terms of some derivatives of the covariance function of $F$.
In particular, it implies that the critical points of an isotropic stationary Gaussian field do not repel each other, only just
(see Corollary \ref{cor:no repulsion} below).
The value of $a_F$ in \eqref{eq: K_2 asymptotic} is of no particular significance other than its mere positivity.
What might be of some interest is the relation between
the density of the critical points $c_F$ and $a_F$. It can show, by comparison with a Poisson process of the same intensity $c_F$, that  critical points are more clustered or dispersed on a small scale compared to the corresponding Poisson process.
Finally, we do believe that the isotropic assumption is not essential for the validity of \eqref{eq: K_2 asymptotic}. There is a couple of reasons to believe in that. The first reason is that the `no attraction' result by Muirhead \cite{Muirhead2019} is indeed valid for generic fields. Our other reasoning lyes within the main structure of our arguments. The isotropic assumption is used to reduce the number of variables and to make the two-point function more amenable to explicit computation. In the isotropic case there are no mysterious cancellations, suggesting the same
for the asymptotic behaviour in the generic case.

\subsection{Statement of the main results}

Our main result concerns the short range asymptotics for the $2$-point correlation function corresponding to smooth stationary isotropic
Gaussian fields $F:\R^{2}\rightarrow\R$. Let $$C_{F}(r)=C_{F}(\|x-y\|):=\E[F(x)\cdot F(y)]$$ be the covariance function of $F$. Assuming that
$F$ is sufficiently smooth and unit variance, and taking into account that $C_{F}$ is even and for every $k\ge 0$,
$$C_{F}^{(2k)}(0) =  (-1)^{k} \E[(\partial_{1}^{k}F(0))^{2}],$$ we may Taylor expand $C_{F}(r)$ around the origin as
\begin{equation}
\label{eq:covar func gen Taylor}
C_{F}(r)=1-g_{2}r^{2}+g_{4}r^{4}-g_{6}r^{6}+O(r^{8}),
\end{equation}
where for all $k\ge 1$, we have
\begin{equation}
\label{eq:g2k def}
g_{2k} = (-1)^{k}\frac{C_{F}^{(2k)}(0)}{(2k)!}\ge 0.
\end{equation}

By rescaling $F$ if necessary, we may further assume w.l.o.g. that $g_{2}=1$ (if $g_{2}$ vanishes it would force $F$
to be a.s. linear, which would contradict $F$ being isotropic, unless $F$ is constant), so that \eqref{eq:covar func gen Taylor} reads
\begin{equation}
\label{eq: C series}
C_{F}(r)=1-r^{2}+g_{4}r^{4}-g_{6}r^{6}+o(r^{6}).
\end{equation}

The following proposition, which is itself an immediate consequence of the Kac-Rice formula \cite[Sections 6.1 and 6.2]{Adler}, can thereupon not be considered as ``new".
\begin{proposition}
For every $R>0$,
\[
\E[\# ( \Cc_{F}\cap B(R))] = \frac{8}{\sqrt 3} g_4 R^2.
\]
\end{proposition}

To simplify the formulas in the main theorem we introduce the following notation:
\begin{equation}
\begin{aligned}
\label{eq:phi def}
\phi(g_{4},g_{6})&:=100g_{4}^{4}-396g_{4}^{2}g_{6}+405g_{6}^{2}>0,\\
\varphi(g_{4},g_{6})&:=-20g_{4}^{4}+88g_{4}^{2}g_{6}-99g_{6}^{2},
\end{aligned}
\end{equation}
and finally
\begin{equation}
\begin{aligned}
\label{eq:AB def}
A(g_{4},g_{6})&:=\sqrt{\varphi(g_{4},g_{6})-(2 g_{4}^{2}- 5g_{6})\sqrt{\phi(g_{4},g_{6})}},\\
B(g_{4},g_{6})&:=\sqrt{-\varphi(g_{4},g_{6})-(2g_{4}^{2}-5g_{6})\sqrt{\phi(g_{4},g_{6})}},
\end{aligned}
\end{equation}
with both $A(g_{4},g_{6})$ and $B(g_{4},g_{6})$ {\em real}.

\begin{theorem}
\label{thm:exp 2 corr main}
Let $F:\R^{2}\rightarrow\R$ be a nonconstant stationary isotropic Gaussian random field, and assume that $F$ is
a.s. $C^{4+\epsilon}(\R^{2})$ for some $\epsilon>0$. Then the $2$-point correlation function corresponding to the critical points admits the following expansion around the origin:
\begin{equation}
\label{eq:K2 exp main}
K_{2}(r) = \frac{\sqrt{3}}{\pi^{2}} \frac{A(g_{4},g_{6})^{2}+B(g_{4},g_{6})^{2}}{\sqrt{\phi(g_{4},g_{6})}}+O_{r\rightarrow 0}(r^{2}),
\end{equation}
where $g_{2k}$ are given by \eqref{eq:g2k def}, and $\phi(\cdot,\cdot)$, $A(\cdot, \cdot)$ and $B(\cdot,\cdot)$ are
given by \eqref{eq:phi def}-\eqref{eq:AB def}.
\end{theorem}

\begin{remark}
One can replace the $C^{4+\epsilon}$ assumption of Theorem \ref{thm:exp 2 corr main} by the more natural $C^{3+\epsilon}$, which would result in the error term 
of $O(r^2)$ in \eqref{eq:K2 exp main} be replaced by $o(1)$.
\end{remark}

\vspace{2mm}

We would like to analyse the leading term in \eqref{eq:K2 exp main}, in particular, whether it may vanish for some values of $g_{4},g_{6}$ that do correspond to some random field, equivalently, whether $A(g_{4},g_{6})^{2}+B(g_{4},g_{6})^{2}$ might vanish. We will show below (see \eqref{eq:A+B=ident}) that $A(g_{4},g_{6})^{2}+B(g_{4},g_{6})^{2}=0$ for some real strictly positive $g_{4},g_{6}$, if and only if
\[
g_{4}^{2}= \frac{5}{2}g_{6}.
\]
To analyse this equation we write the derivatives in terms of the spectral measure $\rho$ of $F$ (that is, $\rho$ is the Fourier transform of the covariance kernel $C(\cdot)$ on $\R^{2}$):
\begin{equation}
\label{eq: Taylor coeff}
C^{(2k)}_F(0) = (-1)^{k}(2\pi)^{2k}\int\limits_{\R^{2}}x_{1}^{2k}d\rho(x).
\end{equation}
By the Cauchy-Schwarz inequality and formulas \eqref{eq:g2k def} and \eqref{eq: Taylor coeff} we obtain the following inequality between $g_4$ and $g_6$:
\begin{equation}
\label{eq:g4, g6 rel}
g_4^2=\frac{(2\pi)^8}{(4!)^2} \left( \int x_1^4d\rho \right)^2 \le \frac{(2\pi)^8}{(4!)^2}\int x_1^6d\rho\int x_1^2d\rho=\frac{5}{2}g_2g_6=\frac{5}{2} g_6.
\end{equation}
The equality holds if and only if $\rho$ is the $\delta$-measure at the origin, equivalently, $F$ is a (random) constant. This means that the leading term is non-zero for non-degenerate $F$.

\begin{corollary}
\label{cor:no repulsion}
Under the assumptions of Theorem \ref{thm:exp 2 corr main} the critical points of $F$ do not repel or attract each other.
\end{corollary}

Our next result is analogous to Theorem \ref{thm:exp 2 corr main},
while separating the critical points into different types: minima, maxima and saddle points.

\begin{theorem}[Separating minima, maxima, saddles] \label{min_max_sadd}
The $2$-point correlation functions corresponding to saddles, local minima, local maxima, and local extrema admit the following expansion around the origin:
\begin{align*}
K_2^{min, min}(r) =O(r^3 \log(1/r)),\;\;\; K_2^{max, max}(r)=O(r^3 \log(1/r)),
\end{align*}
\begin{align*}
 K_2^{saddle, saddle}(r)=O(r^3 \log(1/r)),\;\;\; K_2^{e,e}(r)=O(r^3 \log(1/r)).
\end{align*}
The situation with $K_2^{max, min}(r)$ is a bit more delicate. For a generic field, it is $O(r^3)$, but if the coefficient in front of $r$ in the expansion of the Hessian vanishes (namely $b_{1,1}=0$ in \eqref{eq: b c expansion}), then it is $O(r^7)$. Importantly, it does happen for RWM.

\end{theorem}

In this note we generalize the results of \cite{BCW2017} to a generic class of stationary isotropic fields. Our proofs follow along the general strategy set within \cite{BCW2017}, but involve heavier and more technically challenging computations for performing the asymptotic analysis for the relevant Kac-Rice integrals. 
When this work was complete we have learned of the preprint by Aza\"{i}s and Delmas \cite[Theorem 5.2]{AD}, who obtained a similar result in any dimension by using a different method. Namely, they asymptotically evaluate the same two-point function around the diagonal, by bringing in techniques from Random Matrix Theory, as suggested by Fyodorov \cite{Fyodorov}. Namely, it was observed that the Hessian has the law of the sum of a diagonal matrix and a Gaussian Orthogonal Ensemble (GOE), and an explicit expression for the joint density of GOE eigenvalues is exploited. We believe that in two-dimensional case our method is more transparent and the computation is a less technical compared to \cite{AD}.
In principle, our method could be extended to any fixed dimension, but its computational complexity grows exponentially rapidly, and becomes impractical very soon. 

\subsection{Acknowledgement}
D.B. was partially funded by EPSRC Fellowship EP/M002896/1. D.B. was partially supported by the Ministry of Science and Higher Education of Russia (subsidy in the form of a grant for creation and development of
International Mathematical Centers, agreement no.  075-15-2019-1620, November 8, 2019).
The research leading to these results has received funding from the European Research Council under
the European Union's Seventh Framework Programme (FP7/2007-2013), ERC grant agreement n$^{\text{o}}$ 335141 (I.W.).
V.C. has received funding from the Istituto Nazionale di Alta Matematica (INdAM) through the GNAMPA Research Project 2020.

\section{Expected Number of Critical Points}

Counting the critical points in a ball $B(R) \subseteq \R^{2}$ is equivalent to counting the zeros of the map $x \to \nabla F(x)$. By the Kac-Rice formula the density of critical points is
\begin{equation*}
K_1(x)= \phi_{\nabla F(x)} (0) \cdot \mathbb{E}[ \left| \det{H_{F}(x)} \right| \big| \nabla F(x)=0],
\end{equation*}
where $\phi_{\nabla F(x)}$ is the Gaussian probability density of two-dimensional vector $\nabla F(x)\in \mathbb{R}^2$ evaluated at $0$. By the Kac-Rice formula, if $\nabla F(x)$  is nonsingular for all $x \in B(R)$, then
\begin{equation} \label{zero}
\mathbb{E}[\# (\Cc_{F}\cap B(R)) ]=\int_{B(R)} K_1(x) dx= \vol(B(R)) K_1.
\end{equation}
where in the last step we use the fact that $F$ is assumed isotropic. To write an analytic expressions for $K_1$ we evaluate the covariance matrix $\Sigma$ of the $5$-dimensional centred jointly Gaussian vector $(\nabla F(x), \nabla^2 F(x))$ where $\nabla^2 F(x)$ is the vectorized Hessian evaluated at $x$ (see Appendix \ref{AppA}):
\vspace{-0.1cm}
\begin{equation*}
\Sigma=\left( \begin{matrix}A & B \\ B^t &C \end{matrix} \right),
\end{equation*}
\vspace{-0.1cm}
where
\vspace{-0.1cm}
\begin{equation*}
A =\left( \begin{matrix}  2  &  0 \\   0 & 2  \end{matrix} \right), \hspace{1cm} B=0, \hspace{1cm} C =\left( \begin{matrix}  24 g_4& 0 & 8 g_4   \\  0 & 8 g_4 & 0\\ 8 g_4 & 0 & 24 g_4 \end{matrix} \right).
\end{equation*}
\vspace{-0.1cm}
Now using the value of the matrix $A$,
and thanks to the statistical independence of the first and the second order mixed derivatives of $F$ at every fixed point
$x \in \mathbb{R}^2$, we have
\begin{equation}  \label{uno}
K_1= \frac{1}{2 \pi \sqrt{4 }} \mathbb{E}[ \left| \det{H_{F}(x)} \right| ].
\end{equation}
Using the value of the
covariance matrix $C$ of $\nabla^2 F(x)$, and following the argument in the proof of \cite[Proposition 1.1]{CamWig}, we note that
\vspace{-0.05cm}
\begin{equation} \label{due}
 \mathbb{E}[ \left| \det{H_{F}(x)} \right| ]=8 g_4 \mathbb{E}[ \left|Y_1 Y_3 -Y_2^2 \right| ] = 8 g_4 \frac{2^2}{\sqrt 3},
\end{equation}
\vspace{-0.05cm}
where $(Y_1,Y_2,Y_3)$ is a centred jointly Gaussian random vector with covariance matrix
\begin{equation*}
C =\left( \begin{matrix}  3& 0 & 1  \\  0 & 1 & 0\\ 1 & 0 & 3 \end{matrix} \right).
\end{equation*}
The statement follows combining \eqref{zero}, \eqref{uno} and \eqref{due}:
\begin{equation}
\mathbb{E}[\# (\Cc_{F}\cap B(R)) ]= \vol(B(R)) \frac{1}{2 \pi \sqrt{4}}8 g_4 \frac{2^2}{\sqrt 3}=    \frac{8 }{\sqrt 3} g_4 R^2    .
\end{equation}

\section{Second Factorial Moment}

\subsection{On the Kac-Rice formula for computing the second factorial moment of the number of critical points}~

\smallskip

As explained in the introduction, for $x\ne y$
\begin{align*}
K_2(x,y)=\phi_{(\nabla F(x),\nabla F(y))} (0, 0) \cdot \mathbb{E}[ |\mathrm{det} H_{F}(x)| \cdot |\mathrm{det} H_{F}(y)| \big| \nabla F(x)=\nabla F(y)=0],
\end{align*}
i.e., a Gaussian integral involving the covariance function $C_F$ and its derivatives. This naturally reduces to studying the distribution of the centred Gaussian vector
\begin{equation} \label{repp}
(\nabla F(x),\nabla F(y),\nabla^2 F(x),\nabla^2 F(y))
\end{equation}
with covariance matrix ${\bf \Sigma}(x,y)$, $x,y \in  B(r)$. It is known \cite[Theorem 6.9]{AzWsch} that, if for all $x \ne y$ the Gaussian distribution of $(\nabla F(x),\nabla F(y))$ is non-degenerate, the second factorial moment of the number of critical points in $B(r)$ can be expressed as
\begin{align} \label{18:03}
\mathbb{E}[\# (\Cc_{F}\cap B(r))  \;  (\# (\Cc_{F}\cap B(r)) -1)]=\iint_{B(r) \times B(r)} K_2(x,y) \; d x \, d y.
\end{align}
We note that $K_2$ is everywhere nonnegative.

\subsection{Proof of Theorem \ref{thm:exp 2 corr main}}

\label{subsec: proof of theorem}

\begin{proof}
In order to study the asymptotic behaviour of the second factorial moment of the number of critical points in $B(r)$, as the radius $r$ of the disk goes to zero,  we need to study the centred Gaussian random vector \eqref{repp}.
Its covariance matrix ${\bf \Sigma}={\bf \Sigma}(x,y)$ is of the form
$${\bf \Sigma}=\left(\begin{array}{ccc}
{\bf A}  & {\bf B}\\
 {\bf B}^t& {\bf C}
\end{array} \right),$$
where
${\bf A}={\bf A}(x,y)$ is the covariance matrix of the gradients $(\nabla F(x),\nabla F(y))$, ${\bf C}={\bf C}(x,y)$ is the covariance matrix of the second order derivatives $(\nabla^2 F(x),\nabla^2 F(y))$ and ${\bf B}={\bf B}(x,y)$ is the covariance matrix of the first and second order derivatives.

The function $F$ is isotropic, hence, the law of the critical point process is also invariant w.r.t. translations and rotations. This means that its $2$-point function $K_2(x,y)$ depends on $||x-y||$ only (but not the covariance matrix $\bf \Sigma$); by the standard abuse of notation we write
\begin{equation}
\label{eq:K2=K2(d)}
K_2(x,y)=K_2(||x-y||).
\end{equation}
We will asymptotically evaluate $K_2(x,y)$ for $x=(0,0)$ and $y=(0,r)$ in the relevant regime, which, thanks to the by-product \eqref{eq:K2=K2(d)} of the isotropic property of $F$, will also yield the same for $K_2(r)$.

In Appendix \ref{ucn} the entries of ${\bf \Sigma}(x,y)$ are evaluated for the said $x$ and $y$,
and in Appendix \ref{matrixdelta} the covariance matrix ${\bf \Delta}={\bf \Delta}(x,y)$ of $(\nabla^2 F(x),\nabla^2 F(y))$ conditioned on $\nabla F(x)=\nabla F(y)=0$ is evaluated, i.e.,
\begin{align*}
{\bf \Delta}={\bf C} - {\bf B}^t {\bf A}^{-1} {\bf B}.
\end{align*}
From now on we will only work with ${\bf \Sigma}(r)$ and ${\bf \Delta}(r)$ are defined (not canonically) as ${\bf \Sigma}(x,y)$ and ${\bf \Delta}(x,y)$ with $x=(0,0)$ and $y=(0,r)$.

Since $K_2$ is written in terms of the density and expectation of (a function of) a six-dimensional Gaussian vector, it can be written as a Gaussian integral:
\begin{equation}
\label{k2}
\begin{aligned}
K_2(r)&=\frac{1}{(2 \pi)^2 \sqrt{\text{det}({\bf A}(r))}} \\
&\times \int_{\mathbb{R}^6 }|\zeta_{1} \zeta_{3} - \zeta^2_{2}| \cdot |\zeta_{4} \zeta_{6} - \zeta^2_{5}| \frac{1}{(2 \pi)^3} \frac{1}{\sqrt{\text{det}({\bf \Delta}(r))}} \exp \left\{-\frac{1}{2} \zeta^t {\bf \Delta}^{-1}(r) \zeta \right\} d \zeta,
\end{aligned}
\end{equation}
where $\zeta=(\zeta_1, \zeta_2, \zeta_3, \zeta_4, \zeta_5, \zeta_6)$ is a vector in $\mathbb{R}^6$.
Indeed, the density of $(\nabla F(0,0),\nabla F(0,r))$ at zero is given by $(2\pi)^{-2}(\det({\bf A}(r)))^{-1/2}$, and the integral gives the expectation of $|\mathrm{det} H_{F}(x)| \cdot |\mathrm{det} H_{F}(y)|$ with respect to the Gaussian measure of $(\nabla^2 F(x),\nabla^2 F^2(y))$ conditioned on $\nabla F(x)=\nabla F(y)=0$, that is, having covariance ${\bf \Delta}(r)$.

Our aim is to study the asymptotic behaviour of the $2$-point correlation function $K_2$ in the {\em vicinity} of $r=0$.
When facing an integral of this type, it is useful to transform the coordinates so that rewrite the integrand in terms of the standard Gaussian vector.

From the linear algebra point of view, this is equivalent to the diagonalization of the covariance matrix. For every $r>0$, ${\bf \Delta}(r)$ is symmetric, hence it can be written as
\begin{align} \label{18:07}
{\bf \Delta}(r)={\bf P}^{-1}(r) {\bf \Lambda}(r) {\bf P}(r)={\bf P}^t(r) {\bf \Lambda}(r) {\bf P}(r),
\end{align}
where  ${\bf \Lambda}(r)$ is a diagonal matrix of eigenvalues $\lambda_i(r)$, and ${\bf P}(r)$ is the orthogonal matrix made of the normalized eigenvectors of ${\bf \Delta}(r)$. Note that since ${\bf \Delta}(r)$ is positive definite, the eigenvalues $\lambda_i$ are positive. Using this, we can introduce a new variable
\begin{equation}
\label{eq: zeta}
\xi={\bf \Lambda}^{-1/2}(r) {\bf P}(r) \zeta, \quad \zeta= {\bf Q}(r) {\bf \Lambda}^{1/2}(r)\xi
\end{equation}
where ${\bf Q}={\bf Q}(r)={\bf P}^{-1}(r)={\bf P}^t(t)$.

Both ${\bf \Lambda}$ and ${\bf P}$ can be computed explicitly in terms of the covariance kernel $C_F$. These expressions are not very illuminating. These computations can be found in Lemma \ref{eigenv} and Lemma \ref{eigenvect}.
With this transformation of variables, the Gaussian density in the integral \eqref{k2} becomes the standard Gaussian density $(2\pi)^{-3}\exp(-|\xi|^2/2)$.
The terms $\zeta_{1} \zeta_{3} - \zeta^2_{2}$ and $\zeta_{4} \zeta_{6} - \zeta^2_{5}$ in \eqref{k2} are quadratic forms in $\xi_i$, with coefficients given in terms of $\sqrt{\lambda_i (r)}$ and entries of ${\bf Q}(r)$, whose precise expressions are 
\begin{equation*}
\begin{aligned}
\zeta_{1} \zeta_{3} - \zeta^2_{2} &
=\left( \sum_{j=1}^6 q_{1 j}(r) \sqrt{\lambda_{j}(r) } \; \xi_j \right)
 \left( \sum_{j=1}^6 q_{3 j}(r) \sqrt{\lambda_{j}(r) } \; \xi_j \right)
 \\
 & - \left( \sum_{j=1}^6 q_{2 j}(r) \sqrt{\lambda_{j}(r) } \; \xi_j \right)^2,
\\
\zeta_{4} \zeta_{6} - \zeta_{5}^2 &
= \left( \sum_{j=1}^6 q_{4 j}(r) \sqrt{\lambda_{j}(r) } \; \xi_j \right)
\left( \sum_{j=1}^6 q_{6 j}(r) \sqrt{\lambda_{j}(r) } \; \xi_j \right)
\\
&- \left( \sum_{j=1}^6 q_{5 j}(r) \sqrt{\lambda_{j}(r) } \; \xi_j \right)^2.
\end{aligned}
\end{equation*}
To understand the behaviour {\em around} $r=0$ of $K_2(r)$, we Taylor expand around the origin the entries in ${\bf \Lambda}(r)$ and ${\bf Q}(r)$. Since they are explicitly given in terms of $C_F$ and its derivatives, the first few terms of the Taylor expansion of  ${\bf \Lambda}(r)$ and ${\bf Q}(r)$ can be written in terms of $g_4$ and $g_6$, and the same procedure could be performed for $\det({\bf A}(r))$. 
In what follows the constants involved in the `O'-notation bounding various error terms encountered are {\em absolute}, but may vary from line to line. The main reason is that we are dealing with continuous functions that are homogeneous in $||\xi||$. Their radial (in $\xi$) part is a continuous function on the sphere. Compactness of the sphere allows to write estimates that are asymptotic in $r$ and uniform in the radial part of $\xi$. We explicitly mention the dependence on $||\xi||$. Technical details can be found in Appendices \ref{AppA} and \ref{ucn}. The results of the technically demanding explicit computations, performed in Appendix \ref{ucn}, are as follows.
The matrix $\bf A$ has a simple block structure, and so it is relatively easy to compute its determinant. An explicit computation in Appendix \ref{ucn} gives
$$
\sqrt{\det ({\bf A}(r))}=16 \sqrt 3  g_4 r^2 - 32 \sqrt 3 (g_4^2+g_6) r^4+O(r^6).
$$
For the terms $\zeta_{1} \zeta_{3} - \zeta^2_{2}$ and $\zeta_{4} \zeta_{6} - \zeta^2_{5}$  we have
\begin{align*}
\zeta_{1} \zeta_{3} - \zeta^2_{2}
= -  \frac{24 \sqrt 2 \, \sqrt g_4}{\sqrt{3} \; \phi(g_4,g_6)^{1/4} }   \left[ \xi_3 A(g_4,g_6)+ \xi_4 B(g_4,g_6)\right] \xi_6 \;  r +   (1+ ||\xi||^2)\; O(r^2)
\end{align*}
and
\begin{align*}
\zeta_{4} \zeta_{6} - \zeta_{5}^2
= \frac{24 \sqrt 2 \, \sqrt g_4}{\sqrt{3} \; \phi(g_4,g_6)^{1/4} }   \left[ \xi_3 A(g_4,g_6)+ \xi_4 B(g_4,g_6) \right] \xi_6 \;  r +   (1+ ||\xi||^2)\; O(r^2),
\end{align*}
and for the product, we obtain
\begin{align*}
\left(\zeta_{1} \zeta_{3} - \zeta^2_{2}\right) \cdot \left( \zeta_{4} \zeta_{6} - \zeta_{5}^2 \right)
= -  \frac{384\, g_4}{\sqrt{\phi(g_4,g_6)} }   \left[ \xi_3 A(g_4,g_6)+ \xi_4 B(g_4,g_6) \right]^2\xi_6^2 \;  r^2 +   (1+ ||\xi||^4)\; O(r^4 ),
\end{align*}
where $\phi(g_{4},g_{6})$, $\varphi(g_{4},g_{6})$, $A(g_{4},g_{6})$,
and $B(g_{4},g_{6})$ are defined in \eqref{eq:phi def} and \eqref{eq:AB def}.
Since the covariance matrices are real symmetric, we only manipulate with real numbers, so, in particular, both $A(g_{4},g_{6})$ and $B(g_{4},g_{6})$ are real. Combining these expansions, we get
\begin{equation}
\label{eq:k2 series}
\begin{aligned}
K_2(r)&= \frac{1}{ (2\pi)^5 \sqrt{ \text{det}({\bf A}(r)) }} \left[ \frac{ 384 \, g_4}{\sqrt{\phi(g_4,g_6)}}   \int_{\mathbb{R}^6}
 \left[ \xi_3 A(g_4,g_6) + \xi_4 B(g_4,g_6) \right]^2    \xi_6^2  \right. \\
& \;\; \left. \times   \exp \left\{-\frac{1}{2} \sum_{i=1}^6 \xi^2_i \right\} d \xi \;  r^2+ O(r^4) \right].
\end{aligned}
\end{equation}

The multiple integral in \eqref{eq:k2 series} can be written as a product of one-dimensional integral. Using a standard fact that
\begin{align*}
\int_{\mathbb{R}} \xi^2_i  \exp \left\{-\frac{1}{2}  \xi^2_i \right\} d \xi_i = \sqrt{2 \pi}, \hspace{0.5cm}\int_{\mathbb{R}}  \exp \left\{-\frac{1}{2}  \xi^2_i \right\} d \xi_i = \sqrt{2 \pi},  \hspace{0.5cm}\int_{\mathbb{R}}  \xi_i  \exp \left\{-\frac{1}{2}  \xi^2_i \right\} d \xi_i =0,
\end{align*}
we can rewrite \eqref{eq:k2 series} as
\begin{align*}
&\int_{\mathbb{R}^6} \left[ \xi_3 A(g_4,g_6) + \xi_4 B(g_4,g_6) \right]^2   \xi_6^2   \exp \left\{-\frac{1}{2} \sum_{i=1}^6 \xi^2_i \right\} d \xi \\
&=  \int_{\mathbb{R}^6}  \left[ \xi^2_3 A(g_4,g_6)^{2} + \xi^2_4 B(g_4,g_6)^{2}  + 2 \xi_3 \xi_4 A(g_4,g_6)  B(g_4,g_6) \right]  \xi_6^2  \exp \left\{-\frac{1}{2} \sum_{i=1}^6 \xi^2_i \right\} d \xi  \\
 &=  \int_{\mathbb{R}^6}   \left[ \xi^2_3 A(g_4,g_6)^{2} + \xi^2_4 B(g_4,g_6)^{2} \right]  \xi_6^2  \exp \left\{-\frac{1}{2} \sum_{i=1}^6 \xi^2_i \right\} d \xi \\
 &= [ A(g_4,g_6)^{2} + B(g_4,g_6)^{2} ]  (2 \pi)^3 .
\end{align*}
We finally obtain that, as $r \to 0$,
\begin{align*}
K_2(r)&= \frac{\sqrt 3}{\pi^2} \frac{A(g_4,g_6)^{2} + B(g_4,g_6)^{2}}{\sqrt{ \phi(g_4,g_6)}}   + O(r^2),
\end{align*}
 and, in view of \eqref{18:03}, as  $r \to 0$,
\begin{align*}
\mathbb{E}[\# (\Cc_{F}\cap B(r))  \;  (\# (\Cc_{F}\cap B(r)) -1)]
=  \frac{\sqrt 3}{\pi^2} \frac{A(g_4,g_6)^{2} + B(g_4,g_6)^{2}}{\sqrt{ \phi(g_4,g_6)}}  \pi^2 r^4+O(r^6).
\end{align*}

\end{proof}

We note that
\begin{equation}
\label{eq:A+B=ident}
A(g_4,g_6)^{2} + B(g_4,g_6)^{2}=-(4\, g_4^2 -10 \,  g_6) \sqrt{\phi(g_4,g_6)},
\end{equation}
so $A(g_4,g_6)^{2} + B(g_4,g_6)^{2}=0$, if and only if
\begin{align*}
g^2_4= \frac{5}{2} g_6.
\end{align*}

\section{Proof of Theorem \ref{min_max_sadd}: minima, maxima and saddles}

To prove Theorem \ref{min_max_sadd} we need to evaluate the two-point correlation function $K_2$ modified for the respective types of critical points. The modified function $K_2$ has the same expression \eqref{k2} with the integration over a proper subset of $\mathbb{R}^6$, that is the $\zeta$ are restricted to a domain corresponding to the prescribed type of critical points.

Let us introduce two Hessians at points $x$ and $y$ (already conditioned to be critical points). In terms of $\zeta_i$ these Hessians are given by
\begin{align*}
H_1= \left( \begin{matrix} \zeta_1 & \zeta_2 \\ \zeta_2 & \zeta_3 \end{matrix} \right), \hspace{1cm} H_2= \left( \begin{matrix} \zeta_4 & \zeta_5 \\ \zeta_5 & \zeta_6 \end{matrix} \right).
\end{align*}
The particular type of a critical point depends on the eigenvalues of its Hessian, we may reformulate this dependency in terms of the following quantities:
\begin{equation}
\label{eq:bi=tr, ci=det}
b_i=- \mathrm{Tr} H_i \hspace{1cm} \text{and}  \hspace{1cm}  c_i=\det(H_i).
\end{equation}
A critical point with Hessian $H_i$ is a minimum if $c_i > 0$ and $b_i <0$, a maximum if $c_i >0$ and $b_i >0$, and a saddle if $c_i <0$ (we
ignore the probability $0$ event when one of the eigenvalues vanishes).
As before, we rewrite $\zeta_i$ in terms of $\xi_i$. 
Expanding in powers of $r$ we get
\begin{equation}
\label{eq: b c expansion}
\begin{aligned}
&b_1=-(\zeta_1+\zeta_3)=b_{1,0}+b_{1,1} r + b_{1,2} r^2 +O(r^3),\\
&b_2=-(\zeta_4+\zeta_6)=b_{2,0}+b_{2,1} r + b_{2,2} r^2 +O(r^3), \\
&c_1= \zeta_1 \zeta_3 -\zeta^2_2=c_{1,0}+c_{1,1} r + c_{1,2} r^2 +O(r^3),\\
&c_2= \zeta_4 \zeta_6 -\zeta^2_5=c_{2,0}+c_{2,1} r + c_{2,2} r^2 +O(r^3).
\end{aligned}
\end{equation}
We observe that all the coefficients $b_{i,j}$ are linear functions of the coordinates of $\xi$, and all the coefficients $c_{i,j}$ are quadratic forms of the entries of $\xi$, and also notice that
$$b_{1,0}=b_{2,0}, \; b_{1,1}=-b_{2,1}, \; b_{1,2}=b_{2,2}, \;   c_{1,0}=c_{2,0}=0, \; c_{1,1}=-c_{2,1}, \;c_{1,2}=c_{2,2},$$
where
\begin{align*}
b_{1,0}&=-\frac{8}{\sqrt 3} \sqrt{g_4} \xi_6,\\
b_{1,1}&=3 \sqrt 2 \frac{1}{\phi^{1/4}(g_4,g_6)} [\xi_3 A(g_4,g_6)+\xi_4 B(g_4,g_6)] \\
&\;\;\; +\sqrt 2 \frac{1}{\sqrt{\phi(g_4,g_6)}} \frac{2g_4^2-3 g_6}{|2g_4^2-3 g_6|} \\
&\;\;\; \times \left[-\xi_4 \sqrt{-10 g_4^2+27 g_6+\sqrt{\phi(g_4,g_6)}} \sqrt{\phi(g_4,g_6)+ (8 g_4^2 -18 g_6) \sqrt{\phi(g_4,g_6)} } \right. \\
&\;\;\; \;\;\; \;\; +  \left. \xi_3 \sqrt{-10 g_4^2+27 g_6-\sqrt{\phi(g_4,g_6)}} \sqrt{\phi(g_4,g_6)- (8 g_4^2 -18 g_6) \sqrt{\phi(g_4,g_6)} }  \right],  \\
b_{1,2}&=\frac{2}{\sqrt 3} \frac{g_6}{\sqrt{g_4}} \xi_6 + \frac{1}{\sqrt{g_4}} \xi_5 \sqrt{280 g_4 g_8 - 153 g_6^2}, \\
c_{1,1}&=-8 \sqrt{6}  \frac{\sqrt{g_4}}{\phi^{1/4}(g_4,g_6)} \xi_6 \left[  \xi_4   B(g_4,g_6) +  \xi_3  A(g_4,g_6) \right]\\
c_{1,2}&= 4 (2 g_4^2-9 g_6) \xi_1^2 + 6 \frac{1}{\sqrt{\phi(g_4,g_6)}} \frac{2g_4^2-3 g_6}{|2g_4^2-3 g_6|} \left[ \xi_3 A(g_4,g_6) + \xi_4 B(g_4,g_6) \right] \\
&\;\;\; \times \left[-\xi_4 \sqrt{ -10 g_4^2+27 g_6+\sqrt{\phi(g_4,g_6)}  } \sqrt{\sqrt{\phi(g_4,g_6)} +8 g_4^2 -18 g_6   }  \right. \\
&\;\;\; \;\;\; \;\; +  \left. \xi_3 \sqrt{-10 g_4^2+27 g_6-\sqrt{\phi(g_4,g_6)} )} \sqrt{ \sqrt{\phi(g_4,g_6)} - (8 g_4^2 -18 g_6) }   \right]\\
&\;\;\; + 8 \xi_6 \left[ g_6 \xi_6 + \xi_5 \sqrt{\frac{280}{3} g_4 g_8 -51 g_6^2}\right] .
\end{align*}
Note that $\phi(g_4,g_6)> 0$ (unless $g_4=g_6=0$) and that
$B(g_4,g_6) \ge 0$ (given  $g_{4}^{2}< \frac{5}{2}g_{6}$ ) and equal to zero if and only if $g_4=g_6=0$. We introduce $s_i=\xi_i/|\xi|$, $s\in {\mathcal S}^5$, $|\xi| \in (0,\infty)$, abusing notation we denote with the same letters the rescaled coefficients that are now function of $s_i$ instead of $\xi_i$. Using the above notation and upon integrating out the radial part, the two point function $K_2(R)$ becomes
\begin{equation} \label{k2_s}
K_2(r)= \frac{12}{\pi^5 \sqrt{\det({\bf A}(r))}} \int_{{\mathcal S}^5} |c_1(s) c_2(s)| d s,
\end{equation}
where $c_{j}$ are as in \eqref{eq:bi=tr, ci=det}, and $d s$ is the spherical volume element on the unit sphere ${\mathcal S}^5$.

\subsection{Minimum-minimum}
The two-point correlation function $K^{min,min}_2(r)$ corresponding to the local minima is given by \eqref{k2_s} with integration domain
\begin{align*}
{\mathcal S}_{min, min}&=\{s \in {\mathcal S}^5: \; c_1(s)>0, c_2(s)>0, b_1(s)<0, b_2(s)<0\} \\
&\subseteq \{s \in {\mathcal S}^5: \; c_1(s)>0, c_2(s)>0 \}.
\end{align*}
Since $c_{1,1}=-c_{2,1}$, for some constant $C$ sufficiently big we have
\begin{align*}
{\mathcal S}_{min, min} & \subseteq  \{s \in {\mathcal S}^5: \; |c_{1,1}(s)|<C r\},
\end{align*}
and
\begin{align*}
\int_{{\mathcal S}_{min, min}} |c_1(s) c_2(s)| d s &\le  \int_{  \{s \in {\mathcal S}^5: \; |c_{1,1}(s)|<C r\}} |c_1(s) c_2(s)| d s\\
&=O(r^4)  \int_{  \{s \in {\mathcal S}^5: \; |c_{1,1}(s)|<C r\} }d s=O(r^5 \log(1/r)),
\end{align*}
where the last equality is a mere estimate for the volume of the set where the product of two coordinates is bounded by a quantity of order $r$. This yields
\[
K_2^{min, min}(r)=O(r^3 \log(1/r)).
\]

\subsection{Maximum-maximum}
Similarly we note that the two-point correlation function \\ $K^{max,max}_2(r)$ corresponds to \eqref{k2_s} with integration domain
\begin{align*}
{\mathcal S}_{max, max}&=\{s \in {\mathcal S}^5: \; c_1(s)>0, c_2(s)>0, b_1(s)>0, b_2(s)>0\}\\
& \subseteq \{s \in {\mathcal S}^5: \; c_1(s)>0, c_2(s)>0 \}
\end{align*}
and it immediately shows that, as before, we have
$$K_2^{max, max}(r)=O(r^3 \log(1/r)).$$

\subsection{Saddle-saddle and extremum-extremum}

For two extrema or two saddle points both $c_i$ are forced to be of the same sign, the same argument as for the minimum-minimum case yields
$$K_2^{saddle, saddle}(r)=O(r^3 \log(1/r)), \hspace{1cm} K_2^{e,e}(r)=O(r^3 \log(1/r)).$$

\subsection{Minimum-maximum}
We consider now the integration domain
\begin{align*}
{\mathcal S}_{min, max}&=\{s \in {\mathcal S}^5: \; c_1(s)>0, c_2(s)>0, b_1(s) \text{ and } b_2(s) \text{ of different sign} \},
\end{align*}
as above, $c_{1,1}=-c_{2,1}$, forces $ |c_{1,1}|<C r$ for some constant $C$ sufficiently big. Assuming
$b_{1,1}\ne 0$ and observing that $b_{1,0}=b_{2,0}$, $b_1$ and $b_2$ of different sign implies that $|s_6|< C r$  for some big constant $C$. We have
\begin{align*}
\int_{{\mathcal S}_{min, max}} |c_1(s) c_2(s)| d s& \le  \int_{ \left\{ s \in {\mathcal S}^5: \; |s_6|<Cr \text{ and } |c_{1,1}(s)|<C r  \right\}} |c_1(s) c_2(s)| d s \\
 & \le O(r^4) \int_{ \left\{s \in {\mathcal S}^5: \; |s_6|<Cr \text{ and } |c_{1,1}(s)|<C r\right\}} d s = O(r^5),
\end{align*}
that yields
$$K_2^{min, max}(r)=O(r^3).$$

In the case $b_{1,1}= 0$ (which is the case for RWM), the condition that $b_1$ and $b_2$ are of different sign implies that $|s_6|< C r^2$  for some big constant $C$. Under the assumption $|s_6|< C r^2$ both $b_i$ are of the form
\begin{align*}
b_i=- \frac{8}{\sqrt 3} \sqrt{g_4} s_6+\frac{1}{\sqrt g_4} s_5  \sqrt{280 g_4 g_8 - 153 g_6^2} +O(r^3),
\end{align*}
and again, $b_i$ of different signs, forces the term corresponding to $O(r^3)$ to dominate, that is
$$L(s_5,s_6)=\left| - \frac{8}{\sqrt 3} \sqrt{g_4} s_6+\frac{1}{\sqrt{g_4}} s_5  \sqrt{280 g_4 g_8 - 153 g_6^2}\; r^2 \right| < C r^3.$$
Combining all of this we obtain the estimate
\begin{align*}
\int_{{\mathcal S}_{min, max}} |c_1(s) c_2(s)| d s& \le  \int_{ \left\{ s \in {\mathcal S}^5: \; |s_6|<C r^2 \text{ and } L(s_5,s_6)< C r^3  \right\}} |c_1(s) c_2(s)| d s \\
 & \le O(r^4) \int_{ \left\{s \in {\mathcal S}^5: \; |s_6|<C r^2 \text{ and } L(s_5,s_6)< C r^3 \right\}} d s = O(r^9),
\end{align*}
and
\begin{align*}
K_2^{max, min}(r)=O(r^7).
\end{align*}

\appendix

\section{Covariance matrix of $(\nabla F(x), \nabla^2 F(x))$} \label{cov_exp} \label{AppA}

By translation invariance, the covariance matrix $\Sigma$ of the $5$-dimensional centred Gaussian vector which combines the gradient and the vectorized Hessian evaluated at $x$, does not depend on the point $x \in \mathbb{R}^2$. It is convenient to write $\Sigma$ as a block matrix:
\begin{equation*}
\Sigma=\left( \begin{matrix}A & B \\ B^t &C \end{matrix} \right),
\end{equation*}
where
\begin{equation*}
A=\mathbb{E}[\nabla F(x)^t \cdot \nabla F(y)], \hspace{0.5cm} B=\mathbb{E}[\nabla F(x)^t \cdot \nabla^2 F(y)],\hspace{0.5cm} C=\mathbb{E}[\nabla^2 F(x)^t \cdot \nabla^2 F(y)].
\end{equation*}
The computations of $A$, $B$ and $C$ do not require sophisticated arguments other than iterative differentiation of the covariance function $C_F(||x-y||)$, with covariances of the derivative given by derivatives of the covariance kernel.

As a concrete example of such computation here are the details of the computation for $(A)_{1,1}$. Taking into account the Taylor expansion of $C_F(r)$ around the origin:
\begin{equation*}
C_{F}(||x-y||)=1-\|x-y\|^{2}+g_{4} ||x-y||^{4}-g_{6}||x-y||^{6}+g_{8}||x-y||^{8}+O(||x-y||^{10}),
\end{equation*}
we have
\begin{align*}
(A)_{1,1}= \lim_{ x \to y} \mathbb{E}[ \partial_{x_1} F(x)  \partial_{y_1} F(y) ]=  \lim_{ x \to y} \frac{\partial^2}{\partial_{x_1} \partial_{y_1} } C_{F}(||x-y||)=2.
\end{align*}
With analogous calculations, but using higher order derivatives of $C_{F}(||x-y||)$ we compute the entries of $C$. Finally, since the first and second order derivatives of any stationary field are independent at every fixed point $x \in \mathbb{R}^2$, we immediately have $B = 0$.

\section{Covariance matrix of $(\nabla F(x), \nabla F(x), \nabla^2 F(x), \nabla^2 F(y))$} \label{ucn}
In this section we compute the covariance matrix ${\bf \Sigma}(x,y)$ for the $10$-dimensional Gaussian random vector which combines the gradient and the vectorized Hessian evaluated at $x=(0,0)$ and $y=(0,r)$; thanks to the isotropic property of the field $F$, this is sufficient in order to evaluate $K_2(r)$ for all relevant $r$.
It is convenient to write the matrix $\Sigma(z,w)$ in block form, that is
\begin{equation*}
{\bf \Sigma}(x,y) = \left( \begin{matrix}{\bf A}(x,y) & {\bf B}(x,y) \\ {\bf B}^t(x,y) & {\bf C}(x,y) \end{matrix} \right), \hspace{1cm} \left. {\bf \Sigma}(x,y) \right|_{x=(0,0), y=(0,r)}={\bf \Sigma}(r).
\end{equation*}
The matrices ${\bf A}$, ${\bf B}$ and ${\bf C}$ also have a natural block structure:
\begin{align*}
{\bf  A}(r)=&\left. {\bf A}(x,y) \right|_{x=(0,0), y=(0,r)}= \left( \begin{matrix}A & A(r) \\ A(r) &A \end{matrix} \right),\\
{\bf B}(r)= &\left. {\bf B}(x,y) \right|_{x=(0,0), y=(0,r)}=\left( \begin{matrix}0 & B(r) \\ -B(r) &0 \end{matrix} \right),\\
{\bf C}(r)= &\left. {\bf C}(x,y) \right|_{x=(0,0), y=(0,r)}=\left( \begin{matrix}C & C(r) \\ C(r) &C \end{matrix} \right),
\end{align*}
where $A$ and $C$ are the same as in Section \ref{cov_exp}. $A(r)$ is a diagonal matrix
\begin{equation*}
A(r)=\left( \begin{matrix} \alpha_1(r) & 0 \\ 0 & \alpha_2(r) \end{matrix} \right),
\end{equation*}
the diagonal elements $\alpha_i$ are found by differentiating the covariance kernel of $F$; from the Taylor series for $C_F$ one immediately gets
\begin{align*}
\alpha_1(r)&=\left. \frac{\partial^2}{\partial_{x_1} \partial_{y_1}} C_{F}(||x-y||) \right|_{x=(0,0), y=(0,r)}=2  -4 g_4 r^2 +6 g_6 r^4 +O(r^6),\\
 \alpha_2 (r)&=\left. \frac{\partial^2}{\partial_{x_2} \partial_{y_2}} C_{F}(||x-y||) \right|_{x=(0,0), y=(0,r)}= 2  -12 g_4 r^2 +30  g_6 r^4 +O(r^6),
\end{align*}
so that
\begin{equation*}
\sqrt{\det ({\bf A}(r))}=16 \sqrt 3  g_4 r^2 - 32 \sqrt 3 (g_4^2+g_6) r^4+O(r^6).
\end{equation*}

Analogously to the above, we derive the entries of the matrices $B(r)$ and $C(r)$:
\begin{equation*}
B(r)=\left( \begin{matrix} 0& \beta_1(r) & 0 \\ \beta_1(r) &0& \beta_2(r) \end{matrix} \right),
\end{equation*}
with
\begin{align*}
\beta_1(r)& =-8 g_4 r + 24 g_6 r^3  +O(r^5),\\
\beta_2(r)&  = -24 g_4 r +120 g_6 r^3   +O(r^5),
\end{align*}
and
\begin{equation*}
C(r)=\left( \begin{matrix} \gamma_1(r) & 0 & \gamma_2(r) \\  0 & \gamma_2(r) & 0 \\ \gamma_2(r) &0& \gamma_3(r) \end{matrix} \right),
\end{equation*}
with
\begin{align*}
\gamma_1(r)&= 24  g_4- 72  g_6 r^2 +144  g_8 r^4 +O(r^6) ,\\
\gamma_2(r)&= 8 g_4- 72  g_6 r^2+240   g_8 r^4 +O(r^6),\\
 \gamma_3(r)& = 24 g_4-  360  g_6 r^2 + 1680  g_8 r^4 +O(r^6).
\end{align*}

\section{Conditional covariance matrix}  \label{matrixdelta}
The covariance matrix of the conditional vector
$$( \nabla^2 F(x), \nabla^2 F(y)|  \nabla F(x)= \nabla F(y)=0)$$
is
\begin{equation*}
{\bf \Delta}(r)= {\bf C}(r)-  {\bf B}^t(r)  {\bf A}^{-1}(r)  {\bf B}(r)=\left( \begin{matrix} \Delta_1(r) & \Delta_2(r)\\ \Delta_2(r) & \Delta_1(r)\end{matrix} \right)
\end{equation*}
where $\Delta_1$ and $\Delta_2$ are $3 \times 3$ symmetric matrices such that
\begin{equation*}
\begin{aligned}
 \Delta_1(r)
= \left( \begin{array}{ccc}
 \frac{64}{3} g_4 + a_1(r)&0& a_4(r)\\
0&a_2(r)&0\\
a_4(r)&0&a_3(r)
\end{array}\right), \hspace{0.5cm}
\Delta_2(r)
= \left( \begin{array}{ccc}
\frac{64}{3} g_4+ a_5(r)&0& a_8(r)\\
0&a_6(r)&0\\
a_8(r)&0&a_7(r)
\end{array}\right),
\end{aligned}
\end{equation*}
where
\begin{align*}
\label{eq: series for a}
a_1(r)&=-\frac{2 \beta_1^2(r)}{4 - \alpha_2^2(r)}+\frac{8}{3} g_4,\\
a_2(r)&=- \frac{2 \beta_1^2(r)}{4 - \alpha_1^2(r)}+8 g_4,\\
a_3(r)&= -\frac{2 \beta_2^2(r)}{4  - \alpha_2^2(r)}+ 24 g_4,\\
a_4(r)&= - \frac{2   \beta_1(r) \beta_2(r)}{4 - \alpha_2^2(r)}+ 8 g_4,\\
a_5(r)& =\gamma_1(r)- \frac{ \alpha_2(r) \beta_1^2(r) }{4 - \alpha_2^2(r)}- \frac{64}{3} g_4,\\
a_6(r)&= \gamma_2(r) - \frac{ \alpha_1(r) \beta_1^2(r)}{4 - \alpha_1^2(r)}, \\
a_7(r)&= \gamma_3(r)- \frac{ \alpha_2(r) \beta_2^2(r)}{4 - \alpha_2^2(r)},\\
a_8(r)& = \gamma_2(r) - \frac{ \alpha_2(r) \beta_1(r) \beta_2(r)}{4- \alpha_2^2(r)}.
\end{align*}
From the formulas for elements of ${\bf A}$, ${\bf B}$ and ${\bf C}$, we can obtain explicit formulas and expansions (with the use of
Mathematica software) for $a_i$, note that $a_i(r)=O(r^2)$.

\subsection{Eigenvalue and eigenvectors of ${\bf \Delta}(r)$, $r > 0$} \label{A_i}

Now we compute the eigenvalues and eigenvectors of the matrix ${\bf \Delta}(r)$, for $r > 0$. The following notation is required.

\begin{align*}
A_1^{+}(r)&= a_1(r) + a_5(r)+\frac{2^7}{3} g_4, \hspace{1cm} A_1^{-}(r)= a_1(r) - a_5(r),
\end{align*}
\begin{align*}
A_2^{\pm}(r)&= a_2(r) \pm a_6(r), \\
A_3^{\pm}(r)&= a_3(r) \pm a_7(r), \\
A_4^{\pm}(r)&= a_4(r) \pm a_8(r);
\end{align*}
with $a_i(r)$  defined above.

\begin{lemma} \label{eigenv}
For every $r >0$, the eigenvalues of the matrix ${\bf \Delta}(r)$ have the following explicit expressions:
\begin{equation}
\label{eq:lambda explicit}
\begin{aligned}
\lambda_1(r)&=A_2^{-}(r),\\
\lambda_2(r)&=A_2^{+}(r), \\
\lambda_3(r)&=\frac{1}{2} \Big[A_1^{-}(r) +A_3^{-}(r) - \sqrt{(A_1^{-}(r) - A_3^{-}(r))^2+4 A_4^{-}(r)^2} \Big], \\
\lambda_4(r)&=\frac{1}{2} \Big[A_1^{-}(r) +A_3^{-}(r) + \sqrt{(A_1^{-}(r) - A_3^{-}(r))^2+4 A_4^{-}(r)^2} \Big],\\
\lambda_5(r) &=\frac{1}{2} \Big[A_1^{+}(r) +A_3^{+}(r) - \sqrt{(A_1^{+}(r) - A_3^{+}(r))^2+4 A_4^{+}(r)^2} \Big], \\
\lambda_6(r)&=\frac{1}{2} \Big[A_1^{+}(r) +A_3^{+}(r) + \sqrt{(A_1^{+}(r) - A_3^{+}(r))^2+4 A_4^{+}(r)^2} \Big].
\end{aligned}
\end{equation}
\end{lemma}

\begin{proof}
We can compute explicitly the roots of
$$
\text{det}({\bf \Delta}(r)-\lambda I)=\text{det} \left( \begin{array}{cc} \Delta_1(r) - \lambda I & \Delta_2(r) \\ \Delta_2(r) & \Delta_1(r) -\lambda I
\end{array} \right),
$$
by observing that since $\Delta_i$ are square matrices, we have the following identity for the determinant of a block matrix
$$
\text{det} \left( \begin{array}{cc} \Delta_1(r) - \lambda I & \Delta_2(r) \\ \Delta_2(r) & \Delta_1(r) -\lambda I
\end{array} \right) = \text{det} (\Delta_1(r) - \lambda I - \Delta_2(r)) \, \text{det} (\Delta_1(r) - \lambda I + \Delta_2(r)).
$$
The matrices $\Delta_1(r) - \lambda I \pm \Delta_2(r)$ could be written in terms of $A_i^\pm$ as
\begin{align*}
\Delta_1(r) - \lambda I \pm \Delta_2(r)= \left(\begin{array}{ccc}
A_1^{\pm}(r)-\lambda & 0 & A_4^{\pm}(r) \\
0 & A_2^{\pm}(r) - \lambda &0 \\
A_4^{\pm}(r) & 0 & A_3^{\pm}(r)-\lambda
\end{array} \right).
\end{align*}
Since these matrices have many elements equal to zero, their determinants are particularly simple and could be factorized as
\begin{align*}
\text{det} (\Delta_1(r) - \lambda I \pm \Delta_2(r))= (A_2^{\pm}(r)-\lambda) [\lambda^2 - \lambda (A_1^{\pm}(r)+A_3^{\pm}(r))
+ A_1^{\pm}(r) A_3^{\pm}(r) - A_4^{\pm}(r)^2].
\end{align*}
The last factor is quadratic in terms of $\lambda$ and the roots could be found explicitly. They are equal to $\lambda_5$ and $\lambda_6$ in the ``$+$'' case and $\lambda_3$ and $\lambda_4$ in the ``$-$'' case.
\end{proof}

Lemma \ref{eigenv} expresses the eigenvalues of $\bf \Delta$ in terms of $A_i^\pm$ which, in their turn, are expressed in terms of $a_i$. We compute (with the use of Mathematica software) the asymptotic behaviour of $a_i$. Substituting these expansions into explicit formulas \eqref{eq:lambda explicit} we get expansions for $\lambda_i$ and $\sqrt{\lambda_i}$.\\
After obtaining the explicit formulas for eigenvalues we, again, use computer algebra to find explicit formulas for eigenvectors of $\bf \Delta$. Also the elements of $v_i$ are explicit algebraic expressions in terms of $a_i$:

\begin{lemma} \label{eigenvect}
For every $r>0$, the following vectors $v_i(r)$ are the eigenvectors of the matrix ${\bf \Delta}(r)$ corresponding to $\lambda_i(r)$
\begin{equation*}
\begin{aligned}
 v_1(r)&=(0,-1,0,0,1,0),\\
v_2(r)&=(0,1,0,0,1,0), \\
v_3(r)&=( v_{3 1}(r) , 0,-1,- v_{3 1}(r), 0,1), \\
v_4(r)&=( v_{4 1}(r) ,0,-1,- v_{4 1}(r) ,0,1), \\
v_5(r)&=(- v_{5 1}(r) ,0,1, - v_{5 1}(r) ,0,1), \\
v_6(r)&=(- v_{6 1}(r) ,0,1,- v_{6 1}(r) ,0,1).
\end{aligned}
\end{equation*}
where
\begin{equation*}
\begin{aligned}
v_{3 1}(r)&= \frac{A_3^-(r) - A_1^-(r) + \sqrt{ [A_3^-(r) - A_1^-(r) ]^2 + 4 A_4^-(r)^2 }}{2 A_4^-(r)}, \\
v_{4 1}(r)&= \frac{A_3^-(r) - A_1^-(r) - \sqrt{ [A_3^-(r) - A_1^-(r) ]^2 + 4 A_4^-(r)^2 }}{2 A_4^-(r)}, \\
v_{5 1}(r)&= \frac{A_3^+(r) - A_1^+(r) + \sqrt{ [A_3^+(r) - A_1^+(r) ]^2 + 4 A_4^+(r)^2 }}{2 A_4^+(r)}, \\
v_{6 1}(r)&= \frac{A_3^+(r) - A_1^+(r) - \sqrt{ [A_3^+(r) - A_1^+(r) ]^2 + 4 A_4^+(r)^2 }}{2 A_4^+(r)}.
\end{aligned}
\end{equation*}

\end{lemma}

\end{document}